\documentclass[submission,copyright,creativecommons]{eptcs}
\providecommand{\event}{ACT 2022}
\newcommand{\publicationstatus}{}

\input{lambek.sty}

\title{On the Lambek embedding and the category of product-preserving presheaves}
\author{Peng Fu
\institute{Dalhousie University}
\and
Kohei Kishida
\institute{University of Illinois at Urbana-Champaign}
\and
Neil J. Ross
\institute{Dalhousie University}
\and
Peter Selinger
\institute{Dalhousie University}
}
\def\titlerunning{On the Lambek embedding and the category of product-preserving presheaves}
\def\authorrunning{P. Fu, K. Kishida, N.~J. Ross, P. Selinger}

\begin{document}
\maketitle

\begin{abstract}
  It is well-known that the category of presheaf functors is complete
  and cocomplete, and that the Yoneda embedding into the presheaf
  category preserves products. However, the Yoneda embedding does not
  preserve coproducts. It is perhaps less well-known that if we
  restrict the codomain of the Yoneda embedding to the full
  subcategory of limit-preserving functors, then this embedding
  preserves colimits, while still enjoying most of the other useful
  properties of the Yoneda embedding. We call this modified embedding
  the \textit{Lambek embedding}.  The category of limit-preserving
  functors is known to be a reflective subcategory of the category of
  all functors, i.e., there is a left adjoint
  for the inclusion functor. In the literature, the existence of this
  left adjoint is often proved non-constructively, e.g., by an
  application of Freyd's adjoint functor theorem. In this paper, we
  provide an alternative, more constructive proof of this fact.  We
  first explain the Lambek embedding and why it preserves coproducts.
  Then we review some concepts from multi-sorted algebras and
  observe that there is a one-to-one correspondence between
  product-preserving presheaves and certain multi-sorted term
  algebras. We provide a construction that freely turns any
  presheaf functor into a product-preserving one, hence giving an
  explicit definition of the left adjoint functor of the inclusion.
  Finally, we sketch how to extend our method to prove that the
  subcategory of limit-preserving functors is also reflective.
\end{abstract}

% ----------------------------------------------------------------------
\section{Introduction}

Let $\A$ be a small category. Recall that a \emph{presheaf} over $\A$
is just a functor $F:\A\op\to\Set$. The category $\SetAop$ of
presheaves has many desirable properties. It is complete and
cocomplete, cartesian-closed, and even a topos. If $\A$ is monoidal,
then $\SetAop$ is monoidal closed, where tensor products and
exponentials are given by Day's convolution {\cite{day1970closed}}.
Also, the Yoneda embedding $y : \A \to \SetAop$ is full and
faithful. These properties make the presheaf category a natural
candidate for modeling linear functional programming languages
{\cite{FKRS-model-2022,malherbe2013categorical,RS2017-pqmodel}}.

Although the Yoneda embedding preserves all existing products (and
more generally, limits), it does not preserve coproducts. In some
situations, it is useful to have a version of the Yoneda embedding
that also preserves coproducts. For example, in our work on the
categorical semantics of quantum programming languages
{\cite{FKRS-model-2022}}, we start with a base category that models
quantum operations, which is monoidal but not necessarily monoidal
closed.  In order to account for lambda abstraction (i.e.,
\textit{currying}), we can embed the base category into its presheaf
category. Sometimes the base category already has coproducts, and it is
natural, and often technically necessary, to require the embedding to
also preserve these coproducts. Fortunately, in this situation, there
is a variant of the Yoneda embedding, which we call the
\textit{Lambek embedding} \cite{lambek2006completions}, that achieves
exactly that.

Before we explain the Lambek embedding, let us first recall why the
Yoneda embedding does not preserve coproducts. Let us consider a small
category $\A$ with a distinguished object $I$ and a coproduct $I+I$.
Let $y : \A \to \SetAop$ be the Yoneda embedding, given by
$y(A) = \Hom(-, A)$. We must show that
\[
  y(I) \xrightarrow{y(\inl)} y(I+I) \xleftarrow{y(\inr)} y(I)
\]
is not a coproduct cone in $\SetAop$.  Suppose it is a coproduct
cone. Since we know that
\[
  y(I) \xrightarrow{\inl} y(I)+y(I) \xleftarrow{\inr} y(I)
\]
is also a coproduct cone, there exists a unique isomorphism
$f : y(I+I) \to y(I)+y(I)$ such that the following composition is
identity:
\[
  y(I)+y(I) \xrightarrow{[y(\inl), y(\inr)]} y(I+I) \xrightarrow{~~f~~}
  y(I)+y(I).
\]
The Yoneda lemma states that $\SetAop(y(A),F) \iso F(A)$ for all
$F:\A\op\to\Set$ and $A \in \A$. In particular, we have
\[ \SetAop(y(I+I), y(I)+y(I))
  \iso \Hom(I+I, I) + \Hom(I+I, I).
\]
Since an element in the disjoint union
$\Hom(I+I, I) + \Hom(I+I, I)$ either belongs to the
left component or the right component, it easily follows that
$f \in \SetAop(y(I+I), y(I)+y(I))$ must be a natural
transformation that either maps its entire domain to the left
component of its codomain, or to the right component. This implies
that $f \circ [y(\inl), y(\inr)] \neq \id$.

We now define the Lambek embedding. Let $\SetAopx$ be the full
subcategory of $\SetAop$ consisting of product-preserving
functors. Note that a product in $\A\op$ is a coproduct in $\A$, so
that a product-preserving functor $F:\A\op\to\Set$ is the same thing
as a contravariant functor that maps coproducts of $\A$ to products of
$\Set$. Every functor of the form $y(A) = \Hom(-,A)$ is product-preserving in this sense, because we have
$\Hom(B+C, A) \iso \Hom(B, A) \times \Hom(C, A)$ for all $B, C\in \A$.
Therefore, the image of the Yoneda embedding $y:\A\to\SetAop$ is
entirely contained in the subcategory $\SetAopx$. The \emph{Lambek
  embedding} $\ybar:\A\to\SetAopx$ is defined to be the restriction of
$y$ to this codomain, i.e., the unique functor making the following
diagram commute.
\[
  \begin{tikzcd}
    \A
    \arrow[d, "\ybar"]
    \arrow[dr, "y"]& \\
    {\SetAopx} \arrow[r, "i", hook]& \SetAop
  \end{tikzcd}
\]
We can show that the Lambek embedding preserves all coproducts that
exist in $\A$. For example, if $A+B$ is a coproduct in $\A$, by the
Yoneda lemma, for any $G \in \SetAopx$, we have
\[
  \begin{array}{lll}
    \SetAopx(\ybar(A + B), G)
    & = & \SetAopx(\Hom(-, A + B), G) \\
    & \iso & \SetAop(\Hom(-, A + B), G)  \\
    & \iso & G(A+B) \\
    & \iso^{*} & G(A)\times G(B)\\
    & \iso & \SetAop(\Hom(-, A), G) \times \SetAop(\Hom(-, B), G)\\
    & \iso & \SetAopx(\Hom(-, A), G) \times \SetAopx(\Hom(-, B), G)\\
    & \iso & \SetAopx(\ybar(A) + \ybar(B), G),
  \end{array}
\]
and therefore $\ybar(A + B) \iso \ybar(A) + \ybar(B)$. Note that the
step marked ``*'' uses the assumption that $G$ is product-preserving.

Like the category of presheaves, the full subcategory $\SetAopx$ has
many desirable properties. It is complete and cocomplete, and it is
monoidal closed if $\A$ is monoidal \cite{day1972reflection}. The
proofs of these properties rely on the fact that $\SetAopx$ is a
reflective subcategory of $\SetAop$, i.e., there is an adjunction
$L \dashv i : \SetAopx \to \SetAop$.  The existence of the left
adjoint functor $L$ is not at all obvious. For example, Kennison's
original proof {\cite{kennison1968limit}} uses Freyd's adjoint
functor theorem {\cite{freyd1964abelian}}, which requires a solution
set condition and the axiom of choice. 

In this paper, we give an explicit construction of the left adjoint
functor $L : \SetAop \to \SetAopx$. The methods of this paper are
probably already familiar to researchers who specialize in categorical
algebra. However, we believe that outside this immediate field, the
connection between presheaves and multi-sorted type theories is
perhaps not as well-known as it should be, and that an explicit
description, such as the one we give here, will be beneficial.

The paper is organized as follows. In Section~\ref{sec:background},
we briefly review some concepts from multi-sorted algebras.  In
Section~\ref{sec:theory-a}, we show that there is a one-to-one
correspondence between presheaves and certain multi-sorted
algebras. In Section~\ref{sec:product-preserving}, we consider
multi-sorted algebras corresponding to product-preserving functors.
In Section~\ref{sec:functor-L}, we define a functor $L$ by
constructing a multi-sorted term algebra, and we show that $L$ is the
left adjoint of the inclusion functor $i$. In
Section~\ref{sec:limit-preserve}, we sketch how to extend our method to
limit-preserving functors.

% ----------------------------------------------------------------------
\section{Background on multi-sorted algebras}
\label{sec:background}

The definition of a multi-sorted algebra starts by assuming we are
given a collection of \emph{sorts}, usually denoted $A,B,C$, etc. An
\emph{arity} is an $(n+1)$-tuple of sorts. A \emph{signature} is a set
of \emph{function symbols}, together with an assignment of an arity to
each function symbol. We usually write $f : A_1,\ldots,A_n\to B$ to
indicate that the function symbol $f$ has arity
$\p{A_1,\ldots,A_n,B}$. When the sorts $A_1,\ldots,A_n,B$ are not
important, we also sometimes say that $f$ is an $n$-ary function.
In case $n=0$, if $c$ is a function symbol of arity $\p{B}$, we also
write $c:B$ and call $c$ a \emph{constant symbol} of sort $B$.

\begin{definition}
  \label{def:sigma-algebra}
  Let $\Sigma$ be a signature. A $\Sigma$-algebra $\T$ consists of the
  following data:
  \begin{itemize}
  \item For each sort $A$, a set $\T(A)$. The sets $\T(A)$ are called
    the \emph{carriers} of the algebra.
  \item For each function symbol $f : A_1,\ldots,A_n\to B$ in
    $\Sigma$, a function
    $\T(f) : \T(A_1)\times\ldots\times\T(A_n)\to\T(B)$.
  \end{itemize}
\end{definition}

\begin{definition}
  Let $\cA, \cB$ be $\Sigma$-algebras. A $\Sigma$-homomorphism
  $\phi : \cA \to \cB$ consists of:
  \begin{itemize}
  \item A function $\phi_{A} : \cA(A)\to \cB(A)$ for each sort $A$,
    such that
  \item for each function symbol $f : A_{1}, \ldots , A_{n}\to B \in
    \Sigma$ and all $a_{i} \in \cA(A_{i})$,
    we have \[\phi_{B} (\cA(f)(a_{1},\ldots, a_{n})) = \cB(f)(\phi_{A_{1}}(a_{1}), \ldots, \phi_{A_{n}}(a_{n})).\]
  \end{itemize}  
\end{definition}

One example of a $\Sigma$-algebra is a term algebra. We first define
well-sorted terms. For each sort $A$, we assume that we are given a
countable set $X_{A}$ of \emph{variables}. We further assume that
these sets of variables are pairwise disjoint, i.e., $X_{A}\cap
X_B=\emptyset$ when $A\neq B$.  The set of $\Sigma$-terms is freely
generated from variables and function symbols, in the obvious
sort-respecting way.
\begin{definition}
  The set of \emph{$\Sigma$-terms of sort $A$} is defined inductively
  as follows. We write $\Sigma\vdash t:A$ to mean that $t$ is a
  $\Sigma$-term of sort $A$.
  \[
    \infer{\Sigma\vdash x : A}{x \in X_{A}}
    \qquad
    \infer{\Sigma\vdash f(t_1,\ldots,t_n) : B.}
    {
      f : A_1,\ldots,A_n\to B \in \Sigma
      \qquad \Sigma \vdash t_1 : A_1
      \quad \ldots
      \quad \Sigma \vdash t_n : A_n
    }
  \]
\end{definition}

We say a term is {\em closed} if it contains no variables.  An
\emph{equation} over a signature $\Sigma$ is a triple $\p{s,t,A}$,
where $s,t$ are $\Sigma$-terms of sort $A$. We usually write an
equation as $s\approx t:A$. If $E$ is a set of equations, we write
$E\vdash s\approx t:A$ to mean that the equation $s\approx t:A$
follows from the equations in $E$.

\begin{definition}
  The relation
  $E\vdash s\approx t:A$ is inductively defined by the following rules.
  \[
    \begin{tabular}{c}
      \infer[\rulename{ax}]{E \vdash s\approx t : A}
      {(s\approx t:A)\in E}
      \qquad
      \infer[\rulename{subst}]{E\vdash s[r/x]\approx t[r/x]:B}
      {x\in X_{A}\quad \Sigma\vdash r:A\quad E\vdash s\approx t:B}
      \\[1ex]
      \infer[\rulename{refl}]{E\vdash s \approx s : A}{\Sigma\vdash s : A}
      \qquad
      \infer[\rulename{symm}]{E\vdash s \approx t : B}{E\vdash t \approx s : B}
      \qquad
      \infer[\rulename{trans}]{E\vdash r \approx t : B}{E\vdash r \approx s : B  & E\vdash s \approx t : B}
      \\[1ex]
      \infer[\rulename{cong}]{E\vdash f(s_1,\ldots, s_n) \approx f(t_1,\ldots, t_n)  : B}
      {E\vdash s_i \approx t_i : A_i \text{ for all $i$}\qquad
      f : A_1,\ldots,A_n\to B}
    \end{tabular}
  \]
\end{definition}
In the rule {\rulename{subst}}, $s[r/x]$ denotes the term obtained
from $s$ by replacing all occurrences of the variable $x$ by the term
$r$. This rule ensures that variables are generic, i.e., if an
equation holds for a variable, then it holds for any term. 

\begin{definition}
  Given a signature $\Sigma$, the \emph{open term algebra} $\T =
  \Term(\Sigma)$ is defined as follows: $\T(A)$ is the set of
  $\Sigma$-terms of sort $A$, and for each function symbol
  $f:A_1,\ldots,A_n\to B$ in $\Sigma$, the function
  $\T(f):\T(A_1)\times\ldots\times\T(A_n)\to\T(B)$ is defined by
  $\T(f)(t_1,\ldots,t_n) = f(t_1,\ldots,t_n)$. The \emph{closed term
    algebra} $\CTerm(\Sigma)$ is defined similarly, except that the
  carriers consist only of closed terms.
\end{definition}
If we are also given
a set of equations $E$, we can define the open and closed
\emph{quotient term algebras} $\Term(\Sigma)/E$ and
$\CTerm(\Sigma)/E$, which are defined in the same way except that the
carriers consist of $\approx$-equivalence classes of (open or closed)
terms. The $\rulename{cong}$ rule ensures that the function $\T(f)$ is
well-defined on such equivalence classes.

\begin{definition}
  \label{def:sigma-e-algebra}
  We say that a $\Sigma$-algebra $\T$ \emph{satisfies} an equation
  $s\approx t:A$, or equivalently, that the equation is \emph{valid}
  in $\T$, if for any $\Sigma$-homomorphism $\phi : \Term(\Sigma) \to
  \T$, we have $\phi_{A}(s) = \phi_{A}(t)$. If $E$ is a set of
  equations, we say that $\T$ satisfies $E$ if it satisfies all the
  equations in $E$. If $\T$ is a $\Sigma$-algebra satisfying a set of
  equations $E$, we also say that $\T$ is a
  \emph{$(\Sigma,E)$-algebra}.
\end{definition}

By a multi-sorted \emph{theory}, we mean all of the above data, i.e.,
a collection of sorts, a signature, and a set of equations.  Note that
$\Term(\Sigma)$ and $\CTerm(\Sigma)$ are $\Sigma$-algebras and
$\Term(\Sigma)/E$ and $\CTerm(\Sigma)/E$ are $(\Sigma,E)$-algebras.

% ----------------------------------------------------------------------
\section{The theory associated to a category \texorpdfstring{$\A$}{A}}
\label{sec:theory-a}

In the rest of this paper, we will work with a small category
$\A$. For convenience and without loss of generality, we will work
with the functor category $\SetA$ rather than $\SetAop$ unless
otherwise noted.

\begin{definition}\label{def:theory-a}
  To the category $\A$, we associate a multi-sorted theory as follows.
  The sorts are the objects of $\A$. The signature $\Sigma_{\A}$ contains
  a unary function symbol $f:A\to B$ for every morphism $f:A\to B$ in
  $\A$. The set of equations $E_{\A}$ consists of the following:
  \begin{rulelist}
  \item[id] $\id(x)\approx x:A$, whenever $\id:A\to A$ is an
    identity morphism and $x$ is a variable of sort $A$.
  \item[comp] $(f\circ g)(x)\approx f(g(x)):C$, whenever
    $f:A\to B$ and $g:B\to C$ are morphisms and $x$ is a variable of
    sort $A$.
  \end{rulelist}
\end{definition}

\begin{remark}\label{rem:algebras-functors}
  There is a one-to-one correspondence between
  $(\Sigma_{\A},E_{\A})$-algebras and functors $T:\A\to\Set$.  Indeed,
  if $T : \A \to \Set$ is functor, then $T$ is a $\Sigma_{\A}$-algebra
  because for each $A \in \A$, there is a set $T(A)$. For each
  function symbol $f : A\to B$ in $\A$, since $f$ is also a morphism
  of $\A$, there is a function $T(f) : T(A) \to T(B)$. The equations
  $\rulename{id}$ and $\rulename{comp}$ are satisfied by
  functoriality. Conversely, every $(\Sigma_{\A},E_{\A})$-algebra
  gives rise to a functor, and the two assignments are mutually
  inverse.
\end{remark}

% ----------------------------------------------------------------------
\section{Product-preserving functors}
\label{sec:product-preserving}

Recall from the introduction that we are interested in taking a
functor $F:\A\to\Set$, and constructing another functor
$L(F):\A\to\Set$ that is product-preserving. Rather than requiring
$L(F)$ to necessarily preserve \emph{all} products, we will consider
the slightly more general problem of preserving some
\emph{distinguished} class of products. Therefore, we will assume that
we are given a small category $\A$ and some particular collection
$\Cones$ of product cones in $\A$. We say that a functor is
$\Cones$-product-preserving when it preserves all of the product cones
in $\Cones$.

On one extreme, we may of course assume that $\A$ has all products,
and that $\Cones$ is the set of all product cones in $\A$. In that
case, a $\Cones$-product-preserving functor is just the same thing as
a product-preserving functor in the ordinary sense. On the other
extreme, $\Cones$ could just consist of a single product cone
$A\leftarrow A\times B\rightarrow B$, in which case a functor is
$\Cones$-product-preserving if it preserves just that one product.

For simplicity and ease of exposition, we will assume in the following
that $\Cones$ is a collection of binary product cones. However, the
same construction works for $n$-ary products for any finite $n$. It
also works for infinite products, assuming that we extend the notion
of term algebras to allow terms of infinite arity (which causes no
problems). Product cones for $n=0$, i.e., terminal objects, are of
course also a special case.

We first define a theory of product-preserving functors.

\begin{definition}\label{def:theory-a-cones}
  To a small category $\A$ with a distinguished collection $\Cones$ of
  (binary) product cones, we associate a multi-sorted theory as
  follows.  The signature $\Sigma_{\Cones}$ has the same function
  symbols as $\Sigma_{\A}$, and additionally, for each product cone
  $A\xleftarrow{\fst}C\xrightarrow{\snd} B$ in $\Cones$, we add a
  function symbol
  \[
  \pair : A,B \to C.
  \]
  Note that there are already function symbols for all the morphisms
  of $\A$, including $\fst:C\to A$ and $\snd:C\to B$. The set of equations
  $E_{\Cones}$ has the same equations as $E_{\A}$, and
  additionally, for each product cone
  $A\xleftarrow{\fst}C\xrightarrow{\snd} B$ in $\Cones$, we add the
  following three equations:
  \begin{rulelist}
  \item[fst] $\fst(\pair(x,y))\approx x : A$, whenever $x$ and $y$ are
    variables of sorts $A$ and $B$, respectively.
  \item[snd] $\snd(\pair(x,y))\approx y : A$, whenever $x$ and $y$ are
    variables of sorts $A$ and $B$, respectively.
  \item[pair] $\pair(\fst(z),\snd(z))\approx z:C$, whenever $z$ is a
    variable of sort $C$.
  \end{rulelist}
\end{definition}

We then have:

\begin{remark}\label{rem:algebras-pp-functors}
  There is a one-to-one correspondence between
  $(\Sigma_{\Cones},E_{\Cones})$-algebras and
  $\Cones$-product-pre\-serv\-ing functors $F:\A\to\Set$. The proof is
  basically the same as that of Remark~\ref{rem:algebras-functors}.
  The point is that the equations $\rulename{fst}$,
  $\rulename{snd}$, and $\rulename{pair}$ are exactly what is required
  to ensure that the cone
  $F(A)\xleftarrow{F(\fst)}F(C)\xrightarrow{F(\snd)} F(B)$ is a
  product cone in $\Set$.
\end{remark}

% ----------------------------------------------------------------------
\section{The construction of the functor \texorpdfstring{$L$}{L}}
\label{sec:functor-L}

Given a small category $\A$ with a distinguished collection $\Cones$
of product cones as in the previous section, we write $\SetACones$ for
the full subcategory of $\SetA$ consisting of
$\Cones$-product-preserving functors. The notation $\SetAx$ used in
the introduction is a special case, where $\Cones$ is the
collection of all product cones.

In this section, we will focus on defining the functor
$L:\SetA\to\SetACones$. By Remark~\ref{rem:algebras-pp-functors}, we
know that a product-preserving presheaf corresponds to a
$(\Sigma_{\Cones},E_{\Cones})$-algebra. Thus, given any functor
$F : \A \to \Set$, we must construct a
$(\Sigma_{\Cones},E_{\Cones})$-algebra $L(F)$. This can be done
freely, as we now show.

\begin{definition}\label{def:theory-functor}
  To any functor $F:\A\to\Set$, we associate a multi-sorted theory as
  follows. The sorts, signature, and equations are the same as in
  Definitions~\ref{def:theory-a} and {\ref{def:theory-a-cones}},
  except for the following:
  \begin{itemize}
  \item For each object $A\in\A$ and each element $c\in F(A)$, we add
    a constant symbol $c:A$ to the signature.
  \item For each morphism $f:A\to B\in\A$ and each $c\in F(A)$, let
    $d=F(f)(c)$. We add an equation $d\approx f(c):B$.
  \end{itemize}
  We write the resulting signature and equation as $\Sigma_{F}$ and
  $E_{F}$, respectively.  
\end{definition}

The quotient term algebra $\CTerm(\Sigma_{F})/{E_{F}}$ still satisfies
the equations {\rulename{id}}, {\rulename{comp}}, {\rulename{fst}},
{\rulename{snd}}, and {\rulename{pair}}, and therefore, by
Remark~\ref{rem:algebras-functors}, it is a product-preserving
functor. Hence we define the functor $L:\SetA\to\SetACones$ to be the
following.

\begin{definition}\label{def:L-functor}
  We define the functor $L : \SetA \to \SetACones$ by
  \[
  L(F) = \CTerm(\Sigma_{F})/{E_{F}}.
  \]
\end{definition}

We note that $L$ is a well-defined functor, because for every natural
transformation $\alpha : F\to G \in \SetA$, we have a
$\Sigma_{\Cones}$-homomorphism $L(\alpha) : \CTerm(\Sigma_{F})/{E_{F}}
\to \CTerm(\Sigma_{G})/{E_{G}}$. This homomorphism is defined
inductively on the structure of the terms in
$\CTerm(\Sigma_{F})/{E_{F}}$, using $\alpha$ for the constant
symbols. In other words, we have $L(\alpha)(f(t_1,\ldots,t_n)) =
f(L(\alpha)(t_1),\ldots,L(\alpha)(t_n))$, where $f$ is any function
symbol from $\Sigma_{\Cones}$, and $L(\alpha)(c) = \alpha(c)$, where
$c$ is one of the constant symbols introduced in
Definition~\ref{def:theory-functor}.

\begin{theorem}
  \label{thm:reflection}
  We have an adjunction $L \dashv i$, where $i:\SetA\to\SetACones$ is
  the inclusion functor.
\end{theorem}

\begin{proof}
  First we will define, for any $F : \A \to \Set$, a natural
  transformation $\eta_{F} : F \to i L(F)$, i.e., $\eta_{F} : F \to
  \CTerm(\Sigma_{F})/{E_{F}}$.  For any $A\in \A$ and $a \in F(A)$,
  we define $\eta_{F, A}(a) = a$.  Note that $\eta_{F}$ is natural in
  $A$ because for any $f : A \to B \in \A$ and $a \in F(A)$, we have
  \[
  (\CTerm(\Sigma_{F})/{E_{F}})(f)(\eta_{F, A}(a))
  = f(a)
  \approx b
  = \eta_{F, B}(F(f)(a)),
  \]
  where $b = F(f)(a)$.  Moreover, $\eta$ is natural in $F$ because for
  any $\alpha_{A} : F(A) \to G(A)$ and $a\in F(A)$, we have
  \[
  (L\alpha)_{A}(\eta_{F, A}(a))
  = (L\alpha)_{A}(a)
  = \alpha_{A}(a)
  = \eta_{G, A}(\alpha_{A}(a)).
  \]

  Next, we will show that for any natural transformation $\gamma : F
  \to i(G)$, there exists a unique natural transformation
  $\hat{\gamma} : L(F) \to G$ such that $\hat{\gamma}\circ \eta_{F} =
  \gamma$.  For all $A\in \A$, we define $\hat{\gamma}_{A} :
  (\CTerm(\Sigma_{F})/{E_{F}})(A) \to G(A)$ by induction on the
  structure of the terms in $(\CTerm(\Sigma_{F})/{E_{F}})(A)$.

  \begin{itemize}
  \item $\hat{\gamma}_{A}(a) = \gamma_{A}(a)$ for any $a \in F(A)$. 
  \item $\hat{\gamma}_{B}(f(t)) = G(f)(\hat{\gamma}_{A}(t))$ for any $f : A \to B \in \A, \Sigma_{F} \vdash t : A$.
    
  \item $\hat{\gamma}_{C}(\pair(s, t)) = G(\pair)(\hat{\gamma}_{A}(s),
    \hat{\gamma}_{B}(t)) $ for any cone
    $A\leftarrow C\rightarrow B$ in $\Cones$, $\Sigma_{F} \vdash s :
    A$, and $\Sigma_{F} \vdash t : B$.
  \end{itemize}
  We must show that $\hat{\gamma}_{A} : (\CTerm(\Sigma_{F})/{E_{F}})(A) \to G(A)$ is a well-defined function for all $A\in \A$. We can show this by induction on $E_{F}\vdash s \approx t : B$.
  \begin{itemize}
  \item Case
    \[
    \infer{E_F \vdash b \approx f(a) : B.}{\Sigma_{F} \vdash {a} : A & b = F(f)(a) & f : A \to B \in\A}
    \]
    In this case, we have
    \[\hat{\gamma}_{B}(b) = \gamma_{B}(b) = \gamma_{B}(F(f)(a)) \stackrel{(*)}{=} G(f)(\gamma_{A}(a)) = \hat{\gamma}_{B}(f(a)).\]
    The equality $(*)$ is by naturality of $\gamma$. 

  \item Case
    \[
    \infer{E_F \vdash \fst(\pair(s, t)) \approx s : A.}{\Sigma_F \vdash s : A & \Sigma_F \vdash t : B}
    \]
    We need to show $\hat{\gamma}_{B}(\fst(\pair(s, t))) = G(\fst)(G(\pair)(\hat{\gamma}_{A}(s), \hat{\gamma}_{B}(t))) = \hat{\gamma}_{A}(s)$. This equality holds because $G$ is a product-preserving functor.
  
\item Case
  \[
  \infer{E_F \vdash \pair(\fst(s), \mathsf{snd}(s)) \approx s : C.}{\Sigma_F \vdash s : C}
  \]
  We need to show $\hat{\gamma}_{C}(\pair(\fst(s), \mathsf{snd}(s))) = G(\pair)(G(\fst)(\hat{\gamma}_{C}(s)), G(\mathsf{snd})(\hat{\gamma}_{C}(s))) = \hat{\gamma}_{C}(s)$. This equality holds because $G$ is a product-preserving functor.
  
\item Case
  \[
  \infer{E_F \vdash f(s_1,\ldots,s_n) \approx f(t_1,\ldots,t_n) : B.}
        {E_F \vdash s_1 \approx t_1 : A_1\quad\ldots\quad
          E_F \vdash s_n \approx t_n : A_n}
  \]
  We have
  \[ \hat{\gamma}_{B}(f(s_1,\ldots,s_n))
  \stackrel{(*)}{=} G(f)(\hat{\gamma}_{A_1}(s_1),\ldots,\hat{\gamma}_{A_n}(s_n)) 
  = G(f)(\hat{\gamma}_{A_1}(t_1),\ldots,\hat{\gamma}_{A_n}(t_n))
  = \hat{\gamma}_{B}(f(t_1,\ldots,t_n)),
  \]
  where the equality $(*)$ uses the induction hypothesis.
  
\item All the other cases are proved similarly.
\end{itemize}

Now we show that $\hat{\gamma}$ is a natural transformation. Suppose
$f : A \to B \in \A$. We need to show
\[G(f)(\hat{\gamma}_{A}(t)) = \hat{\gamma}_{B}((\CTerm(\Sigma_{F})/{E_{F}})(f)(t)).\]
This is true because by the definition of $\hat{\gamma}$,
we have \[\hat{\gamma}_{B}((\CTerm(\Sigma_{F})/{E_{F}})(f)(t)) = \hat{\gamma}_{B}(f(t)) = G(f)(\hat{\gamma}_{A}(t))\] 
for any $\Sigma_{F} \vdash t : A$. 

It is obvious to verify that $\hat{\gamma}\circ \eta = \gamma$.

Lastly, we must show that $\hat{\gamma}$ is unique. Consider a natural
transformation $\hat{\gamma}\,'$ such that $\hat{\gamma}\,' \circ \eta
= \gamma$. We must show $\hat{\gamma}\,'_{A}(t) = \hat{\gamma}_{A}(t)$
for all $\Sigma_{F} \vdash t : A$. This can be shown by induction on
$\Sigma_{F} \vdash t : A$.
\begin{itemize}
\item Case
  \[
  \infer{\Sigma_{F}\vdash a : A.}{a \in F(A)}
  \]
  In this case, we have $\hat{\gamma}\,'_{A}(a) =
  \hat{\gamma}\,'_{A}(\eta_{F, A}(a)) = \gamma_{A}(a) =
  \hat{\gamma}_{A}(a)$.
  
\item Case
  \[
  \infer{\Sigma_{F}\vdash \pair(s, t) : C.}{\Sigma_{F}\vdash s : A & \Sigma_{F}\vdash t : B}
  \]
  By induction hypothesis, we have $\hat{\gamma}\,'_{A}(s) =
  \hat{\gamma}_{A}(s)$ and $\hat{\gamma}\,'_{B}(t) =
  \hat{\gamma}_{B}(t)$.  We need to show
  \[
  \hat{\gamma}\,'_{A\times B}(\pair(s, t)) = \hat{\gamma}_{A\times B}(\pair(s, t)) = G(\pair)(\hat{\gamma}_{A}(s), \hat{\gamma}_{B}(t)) = G(\pair)(\hat{\gamma}\,'_{A}(s), \hat{\gamma}\,'_{B}(t)).
  \]
  This is the case because by the naturality of $\hat{\gamma}\,'$, we have
  \[
  G(\fst)(\hat{\gamma}\,'_{C}(\pair(s, t))) = \hat{\gamma}\,'_{A}(\fst(\pair(s, t)))) = \hat{\gamma}\,'_{A}(s)
  \]
  and
  \[
  G(\mathsf{snd})(\hat{\gamma}\,'_{C}(\pair(s, t))) = \hat{\gamma}\,'_{B}(t).
  \]
  This shows that $\hat{\gamma}\,'_{C}(\pair(s, t))$ is indeed the pair $G(\pair)(\hat{\gamma}\,'_{A}(s), \hat{\gamma}\,'_{B}(t))$.

\item Case
  \[
  \infer{\Sigma_{F}\vdash f(t) : B.}
  {\Sigma_{F} \vdash t : A &
    f : A\to B
  }
  \]
  By naturality of $\hat{\gamma}\,', \hat{\gamma}$ and the induction
  hypothesis, we have
  \[
  \hat{\gamma}\,'_{B}(f(t)) =
  G(f)(\hat{\gamma}\,'_{A}(t)) = G(f)(\hat{\gamma}_{A}(t)) =
  \hat{\gamma}_{B}(f(t)).\qedhere
  \]
\end{itemize}
\end{proof}

\begin{remark}
  In the introduction, we showed that the Lambek embedding
  $\ybar:\A\to\SetAopx$ preserves all coproducts, where $\SetAopx$ is
  the subcategory of $\SetAop$ consisting of product-preserving
  functors. This result can easily be relativized, without changing
  the proof, to the case of a chosen collection of cones. Namely, if
  $\Cones$ is a collection of product cones in $\A\op$ (or,
  equivalently, coproduct cones in $\A$), then the Lambek embedding
  $\ybar:\A\to\SetAopCones$ preserves all of the coproduct cones in
  $\Cones$. Thus, the construction can be customized to preserve
  coproducts ``of interest''.
\end{remark}

% ----------------------------------------------------------------------
\section{The subcategory of limit-preserving presheaves}
\label{sec:limit-preserve}

Our method of showing that the subcategory of
$\mathcal{C}$-product-preserving functors is reflective can also be
adapted to functors that preserve a given class of limits (not
necessarily products). A small complication is that, in Freyd's
terminology, the notion of general limits is not \emph{algebraic} but
only \emph{essentially algebraic} {\cite{freyd1972aspects}}.  This
means that the domain of some operations is defined in terms of
equations. In the following, we sketch the proof in the case of
equalizers, but the same method works for other kinds of limits as
well.

\begin{definition}
\label{def:theory-eq-cones}
  Consider a small category $\A$ with a distinguished collection
  $\Cones$ of limits (we focus on equalizers for simplicity).  To
  this, we associate a multi-sorted theory as follows. The signature
  $\Sigma_{\Cones}$ has the same function symbols as $\Sigma_{\A}$,
  and additionally, for each equalizer $E \stackrel{e}{\to} A
  \stackrel{f,g}{\rightrightarrows} B$ in $\Cones$, we add a function
  symbol
  \[
  \eql : A \to E.
  \]
  Its term formation rule is slightly different than that of other
  function symbols (and for this reason, the resulting theory is not
  strictly speaking a multi-sorted algebraic theory in the sense of
  Section~\ref{sec:background}, but rather what should be called a
  multi-sorted \emph{essentially algebraic} theory):
  \[
    \infer{\Sigma_{\Cones} \vdash \eql(t) : E}
    {\Sigma_{\Cones}\vdash t : A  
      & f,g : A\to B 
      & E_{\Cones} \vdash f(t) \approx g(t):B
    }
  \]
  The set of equations $E_{\Cones}$ has the same equations as
  $E_{\A}$, and additionally, for each equalizer $E \stackrel{e}{\to}
  A \stackrel{f,g}{\rightrightarrows} B$ in $\Cones$, we add the
  following two equations:
  \begin{rulelist}
  \item[beta] $e(\eql(t))\approx t : A$, whenever
    $\Sigma_{\Cones}\vdash t:A$ and $E_{\Cones} \vdash f(t) \approx g(t) : B$.
  \item[eta] $x \approx \eql(e(x)) : E$, whenever $x$ is a
    variable of sort $E$.
  \end{rulelist}  
\end{definition}

Note that Definition~\ref{def:theory-eq-cones} introduces an apparent
circularity, because unlike Definition~\ref{def:theory-a} and
{\ref{def:theory-a-cones}}, the well-sortedness judgement
$\Sigma_{\Cones}\vdash t:E$ and the set of equations $E_{\Cones}$ are
now defined in terms of each other. Of course this is not an actual
circularity; it just means that these two items are defined by
simultaneous induction. Similarly, the notion of a
$(\Sigma_{\Cones},E_{\Cones})$-algebras must be adjusted so that
$\T(\eql):\T(A)\to\T(E)$ is a partial function that is defined and
satisfies $\T(e)(\T(\eql)(x))=x$ for those elements $x\in\T(A)$
satisfying $\T(f)(x)=\T(g)(x)$. With these adjustments, we have the
following:

\begin{remark}\label{rem:algebras-lp-functors}
  Similarly to Remark~\ref{rem:algebras-pp-functors}, we can show that
  there is a one-to-one correspondence between
  $(\Sigma_{\Cones},E_{\Cones})$-algebras and
  $\Cones$-limit-preserving functors $F:\A\to\Set$. The equations
  $\rulename{beta}$ and $\rulename{eta}$ ensure that
  \[ F(E)
  \stackrel{F(e)}{\to} F(A)
  \stackrel{F(f),F(g)}{\rightrightarrows} F(B)
  \]
  is an equalizer in $\Set$. To show that the full subcategory
  $\SetACones$ of $\Cones$-limit-preserving functors is reflective, we
  first define the functor $L : \SetA \to \SetACones$, similarly to
  Definition~\ref{def:L-functor}.  Then we can show that $L$ is a left
  adjoint of the inclusion functor, by adapting the proof of
  Theorem~\ref{thm:reflection}.
\end{remark}

% ----------------------------------------------------------------------
\section{Conclusion}

We gave a brief introduction to the Lambek embedding, a version of the
Yoneda embedding that preserves coproducts (or more generally, a
chosen class of distinguished coproducts or colimits). The Lambek
embedding is obtained by restricting the codomain of the Yoneda
embedding to a suitable full subcategory of $\SetAop$, namely, the
full subcategory of product-preserving functors (or more generally,
functors that preserve the distinguished class of products or
limits). This is a reflective subcategory of $\SetAop$. Like $\SetAop$
itself, this subcategory is complete and cocomplete, as well as
monoidal closed provided that $\A$ is monoidal. Our method uses
concepts from multi-sorted algebras. In particular, we observed that
there is a one-to-one correspondence between product-preserving
functors and certain multi-sorted algebras.  We gave a direct
syntactic construction of the functor $L$ and proved that it is left
adjoint to the inclusion functor.

% ----------------------------------------------------------------------
\section*{Acknowledgements}

This work was supported by the Natural Sciences and Engineering
Research Council of Canada (NSERC) and by the Air Force Office of
Scientific Research under Award No.\@ FA9550-21-1-0041.

\bibliographystyle{eptcs}
\bibliography{lambek}

\begin{thebibliography}{1}
\providecommand{\bibitemdeclare}[2]{}
\providecommand{\surnamestart}{}
\providecommand{\surnameend}{}
\providecommand{\urlprefix}{Available at }
\providecommand{\url}[1]{\texttt{#1}}
\providecommand{\href}[2]{\texttt{#2}}
\providecommand{\urlalt}[2]{\href{#1}{#2}}
\providecommand{\doi}[1]{doi:\urlalt{http://dx.doi.org/#1}{#1}}
\providecommand{\eprint}[1]{arXiv:\urlalt{https://arxiv.org/abs/#1}{#1}}
\providecommand{\bibinfo}[2]{#2}

\bibitemdeclare{inproceedings}{day1970closed}
\bibitem{day1970closed}
\bibinfo{author}{Brian \surnamestart Day\surnameend} (\bibinfo{year}{1970}):
  \emph{\bibinfo{title}{On closed categories of functors}}.
\newblock In: {\sl \bibinfo{booktitle}{Reports of the Midwest Category Seminar
  IV}}, {\sl \bibinfo{series}{Springer Lecture Notes in Mathematics}}
  \bibinfo{volume}{137}, pp. \bibinfo{pages}{1--38}, \doi{10.1007/BFb0060438}.

\bibitemdeclare{article}{day1972reflection}
\bibitem{day1972reflection}
\bibinfo{author}{Brian \surnamestart Day\surnameend} (\bibinfo{year}{1972}):
  \emph{\bibinfo{title}{A reflection theorem for closed categories}}.
\newblock {\sl \bibinfo{journal}{Journal of Pure and Applied Algebra}}
  \bibinfo{volume}{2}(\bibinfo{number}{1}), pp. \bibinfo{pages}{1--11},
  \doi{10.1016/0022-4049(72)90021-7}.

\bibitemdeclare{book}{freyd1964abelian}
\bibitem{freyd1964abelian}
\bibinfo{author}{Peter~J. \surnamestart Freyd\surnameend}
  (\bibinfo{year}{1964}): \emph{\bibinfo{title}{Abelian categories}}.
\newblock \bibinfo{publisher}{Harper \& Row New York}.

\bibitemdeclare{article}{freyd1972aspects}
\bibitem{freyd1972aspects}
\bibinfo{author}{Peter~J. \surnamestart Freyd\surnameend}
  (\bibinfo{year}{1972}): \emph{\bibinfo{title}{Aspects of topoi}}.
\newblock {\sl \bibinfo{journal}{Bulletin of the Australian Mathematical
  Society}} \bibinfo{volume}{7}(\bibinfo{number}{1}), pp.
  \bibinfo{pages}{1--76}, \doi{10.1017/S0004972700044828}.

\bibitemdeclare{unpublished}{FKRS-model-2022}
\bibitem{FKRS-model-2022}
\bibinfo{author}{Peng \surnamestart Fu\surnameend}, \bibinfo{author}{Kohei
  \surnamestart Kishida\surnameend}, \bibinfo{author}{Neil~J. \surnamestart
  Ross\surnameend} \& \bibinfo{author}{Peter \surnamestart Selinger\surnameend}
  (\bibinfo{year}{2022}): \emph{\bibinfo{title}{A biset-enriched categorical
  model for {Proto-Quipper} with dynamic lifting}}.
\newblock \bibinfo{note}{Preprint available from \arxiv{2204.13039}}.

\bibitemdeclare{article}{kennison1968limit}
\bibitem{kennison1968limit}
\bibinfo{author}{J.~F. \surnamestart Kennison\surnameend}
  (\bibinfo{year}{1968}): \emph{\bibinfo{title}{On limit-preserving functors}}.
\newblock {\sl \bibinfo{journal}{Illinois Journal of Mathematics}}
  \bibinfo{volume}{12}(\bibinfo{number}{4}), pp. \bibinfo{pages}{616--619},
  \doi{10.1215/ijm/1256053963}.

\bibitemdeclare{book}{lambek2006completions}
\bibitem{lambek2006completions}
\bibinfo{author}{Joachim \surnamestart Lambek\surnameend}
  (\bibinfo{year}{1966}): \emph{\bibinfo{title}{Completions of categories:
  Seminar lectures given 1966 in Z{\"u}rich}}.
\newblock {\sl \bibinfo{series}{Lecture Notes in
  Mathematics}}~\bibinfo{volume}{24}, \bibinfo{publisher}{Springer},
  \doi{10.1007/BFb0077265}.

\bibitemdeclare{phdthesis}{malherbe2013categorical}
\bibitem{malherbe2013categorical}
\bibinfo{author}{Octavio \surnamestart Malherbe\surnameend}
  (\bibinfo{year}{2013}): \emph{\bibinfo{title}{Categorical Models of
  Computation: Partially Traced Categories and Presheaf Models of Quantum
  Computation}}.
\newblock Ph.D. thesis, \bibinfo{school}{University of Ottawa, Department of
  Mathematics and Statistics}.
\newblock \bibinfo{note}{Available from \arxiv{1301.5087}}.

\bibitemdeclare{inproceedings}{RS2017-pqmodel}
\bibitem{RS2017-pqmodel}
\bibinfo{author}{Francisco \surnamestart Rios\surnameend} \&
  \bibinfo{author}{Peter \surnamestart Selinger\surnameend}
  (\bibinfo{year}{2018}): \emph{\bibinfo{title}{A categorical model for a
  quantum circuit description language. {Extended} Abstract}}.
\newblock In: {\sl \bibinfo{booktitle}{Proceedings of the 14th International
  Conference on Quantum Physics and Logic, QPL 2017, Nijmegen}}, {\sl
  \bibinfo{series}{Electronic Proceedings in Theoretical Computer Science}}
  \bibinfo{volume}{266}, pp. \bibinfo{pages}{164--178},
  \doi{10.4204/EPTCS.266.11}.

\end{thebibliography}

\end{document}